\documentclass[11pt, reqno]{amsart}
\usepackage{amssymb,latexsym,amsmath,amsfonts}
\usepackage{latexsym}
\usepackage{latexsym}
\usepackage{hyperref}
\usepackage[mathscr]{eucal}
\usepackage[english]{babel}
\newtheorem{thm}{Theorem}
\newtheorem{lem}{ \bf Lemma}
\newtheorem{cor}{ \bf Corolarry}
\newtheorem{prop}{\bf Proposition}
\newcommand{\be}{\begin{equation}}
\newcommand{\ee}{\end{equation}}
\newcommand{\Bea}{\begin{eqnarray*}}
\newcommand{\Eea}{\end{eqnarray*}}
\newcommand{\bea}{\begin{eqnarray}}
\newcommand{\eea}{\end{eqnarray}}

\numberwithin{equation}{section}

\def\dg{{\delta}_g}

\def\Rt{{\check{R}}}
\def\la{\Delta}

\def\Rc{\stackrel{\circ}{R}}
\def\Lc{\stackrel{\circ}{\Lambda}}
\def\Wc{\stackrel{\circ}{W}}

\theoremstyle{definition}

\theoremstyle{remark}

\numberwithin{equation}{section}

\begin{document}

\title[The Stability of $L^p$-norms of Riemannian curvature ]{On the stability of $L^p$-norms of Riemannian curvature at rank one symmetric spaces}

\author{Soma Maity}
\address{Indian Institute of Science Education and Research , Mohali, India}
\email{soma123maity@gmail.com}


\begin{abstract}
We study stability and local minimizing properties of $L^p$- norms of Riemannian curvature tensor denoted by $\mathcal{R}_p$ by variational methods. We compute the Hessian of $\mathcal{R}_p$ at compact rank $1$ symmetric spaces and prove that they are stable for $\mathcal{R}_p$ for certain values of $p\geq2$. A similar result also holds for compact quotients of rank $1$ symmetric spaces of non-compact type. Consequently, we obtain stability of $L^{\frac{n}{2}}$- norm of Weyl curvature at these metrics using results from \cite{GV}.

\end{abstract}
\maketitle

\section{Introduction}  Let $M^n$ be a compact smooth manifold without boundary with dimension $n\geq 3$. Given a Riemannian metric $g$ on $M$, the Riemannian curvature tensor, Ricci curvature, scalar curvature and Weyl curvature of $(M,g)$ are denoted by $R_g, r_g, s_g $ and $ W_g$ respectively. It is a classical problem in geometry to study topology and geometry of critical metrics for the following Riemannian functionals defined on the space of Riemannian metrics $\mathcal{M}$ on $M.$ 
\Bea \mathcal{R}_p(g)=\int_M|R_g|^pdv_g, \ \ \mathcal{R}ic_p(g)=\int_M |r_g|^pdv_g, \ \ \mathcal{S}(g)=\int_M s_gdv_g
\Eea
$$\mathcal{S}_p(g)=\int_M|s_g|^pdv_g, \ \ \mathcal{W}_p(g)=\int_M |W_g|^pdv_g$$
where $dv_g$ denotes the volume form of $g$ and $p\geq 2$. We consider $C^{k,\alpha}$-topology on $\mathcal{M}$ for sufficiently large $k$ such that the above functionals have required smoothness. These functionals are not scale invariant unless $p=\frac{n}{2}.$ Hence they are restricted to the space of Riemannian metrics with unit volume.  Irreducible symmetric spaces are critical metrics for all of them. 

Critical metrics of $\mathcal{S}$ are Einstein metrics. Rigidity and stability of Einstein metrics at compact irreducible symmetric spaces have been studied by Koiso in \cite{KN}. Besson, Courtoise, Gallot proved that compact quotients of rank one symmetric spaces of non-compact type are global minima for $\mathcal{S}_{\frac{n}{2}}$ in \cite{BCG}. Gursky and Viaclosky studied rigidity, stability  and local minimizing properties of $\mathcal{F}_t=\mathcal{R}ic_2+t\mathcal{S}_2$ for $t\in \mathbb{R} $, at Einstein metrics in \cite{GV}.

In this paper we study stability and local minimizing properties of $\mathcal{R}_p.$ Convergence and collapsing of metrics under bound on $\mathcal{R}_p$ have been studied in \cite{AM}, \cite{GZ}, \cite{YD}. If the Euler characteristic of $M$ is non-zero then $\inf_{g\in \mathcal{M}_1} \mathcal{R}_p$ is positive for $p>\frac{n}{2}$ \cite{CG1}, \cite{CG2}, \cite{YD}. It is a hard problem to prove that a certain metric is a global minima for this functional. We have studied local minimizing property of $\mathcal{R}_p$ by studying its stability at rank one symmetric spaces for $p\geq 2$. 

Let $\mathcal{S}^2M$ denote the space of symmetric $2$-tensors on $M.$ If $g$ is an Einstein metric which is not a sphere then  $\mathcal{S}^2M$ decomposes as (Lemma (4.57)in \cite{BA})
$$ \mathcal{S}^2M={\rm Im}\dg^* \oplus C^{\infty}(M).g \oplus(\dg^{-1}(0)\cap {\rm tr}_g^{-1}(0))$$ 
where $\dg^*$ is the formal adjoint of the divergence operator $\dg$ on $\mathcal{S}^2M.$ A trace free and divergence free symmetric two tensor is called a {\it transverse-traceless} tensor or a TT-tensor. Let $H_p$ denote the Hessian of $\mathcal{R}_p$ at a critical metric $g$. Since $\mathcal{R}_p$ remains invariant under the action of the group of diffeomorphisms of $M$, $H_p$ vanishes on ${\rm Im}\dg^*$. 
\\
\\
{\bf Definition:} Let $(M,g)$ be a critical point for $\mathcal{R}_p$. $(M,g)$ is said to be {\it stable} for $\mathcal{R}_p$ if there exists an $\epsilon>0$ such that 
\be H_p(h,h)\geq \epsilon \|h\|^2 \ \ \forall \ h\in (\dg^{-1}(0)\cap {\rm tr}_g^{-1}(0))\oplus \{fg: \int_M f dv_g=0\}
\ee
where  $\|.\|$ is the $L^2$-norm on $\mathcal{S}^2M$ induced from $g.$
\begin{thm}\label{main} Let $(M,g)$ be a compact rank one symmetric space with unit volume. $(M,g)$ is stable for $\mathcal{R}_p$ for the following values of $p.$

(i) $\mathbb{C}P^m$ and $p\in [2,2m]$

(ii) $\mathbb{H}P^m$ and $p\in [2m,4m]$

(iii) $\mathbb{O}P^2$ and $p\in [9,16]$

where $\mathbb{C}P^m$,  $\mathbb{H}P^m$, $\mathbb{O}P^2$ denote Complex projective space, Quaternionic projective space and the Caley plane respectively.
\end{thm}
\begin{thm}\label{main2} Let $(M^n, g)$ be a compact manifold with unit volume and dimension $n\geq 3$. If the Riemannian universal cover of $(M,g)$ is a rank one symmetric space of non-compact type then $(M^n,g)$ is stable for $\mathcal{R}_p$ for $p\geq \frac{n}{2}.$
\end{thm}
Similar results have been proved for space forms and their products in \cite{MY}, \cite{SM1}. A formula for the gradient of $\mathcal{R}_2$ is (Proposition (4.70) in \cite{BA}) 
\be \nabla\mathcal{R}_2(g)=2\delta^Dd^Dr-2\Rt+\frac{1}{2}|R|^2g+(\frac{2}{n}-\frac{1}{2})\|R\|^2g .\ee
We refer to section 2 for notations. We compute $H_2$ at an Einstein critical metric of $\mathcal{R}_2$ using the above formula in Section 3.  Using curvature of rank one symmetric spaces we prove that $H_2$ restricted to transverse traceless tensors satisfies the stability condition given in (1.1). If $(M,g)$ is a symmetric space then $H_p$ is a constant multiple of $H_2$ when they are restricted to TT-tensors. Hence $H_p$ restricted to TT-tensors also satisfies (1.1). 

To study $H_p$ on the space of conformal variations of $g$ we use Proposition (1.1) in \cite{SM2}. We observe that if $f$ is an eigenfunction corresponding to the first positive eigenvalue of the Laplacian of $\mathbb{H}P^m$ or $\mathbb{O}P^2$ then $H_2(fg,fg)$ is negative. Hence they are not stable for $\mathcal{R}_2.$  But if $p$ is sufficiently large then $H_p$ is positive for certain values of $p\in [2,n]$ for all compact rank one symmetric spaces. If $(M,g)$ is a compact quotient of a rank one symmetric space of non-compact type then $H_p$ restricted to conformal variations of $g$ is positive for all $p\geq \frac{n}{2}.$ 

From techniques in \cite{GV} we observe that if $(M,g)$ is stable then it is a strict local minimizer for $\mathcal{R}_p.$ 
\begin{thm}\label{R_minima}
Let $(M^n,g)$ be a compact locally symmetric space with dimension $n\geq 3$. $(M^n,g)$ is a strict local minimizer for $\mathcal{R}_p$ in the following cases.

(i) $\mathbb{C}P^m$ and $p\in [2,2m]$

(ii) $\mathbb{H}P^m$ and $p\in [2m,4m]$

(iii) $\mathbb{O}P^2$ and $p\in [9,16]$

(iv) If the Riemannian universal cover of $(M^n, g)$ is a non compact rank $1$ symmetric space then $p\geq \frac{n}{2}.$
\end{thm}
 It is conjectured in \cite{GV} that $\mathbb{H}P^m$ and $\mathbb{O}P^2$ are strict local minima for $\mathcal{R}ic_2$. The stability of $\mathcal{R}ic_2$ depends on existence of a suitable lower bound for the first eigenvalue of the Lichnerowicz Laplacian in these cases. Whereas the stability of $\mathcal{R}_p$ at $\mathbb{H}P^m$ and $\mathbb{O}P^2$ does not require such bound.
 
From the decomposition of algebraic curvature tensors we have,
\be R_g = \frac{s}{n(n-1)}g\wedge g +\frac{2}{n-2}(r-\frac{s}{n}g)\wedge g +W_g\ee
If $fg$ is a metric conformal to $g$ then the above decomposition implies that $W_{fg}=fW_g$. Hence, $\mathcal{W}_{\frac{n}{2}}(fg)=\mathcal{W}_{\frac{n}{2}}(g)$.  We prove that 
\begin{thm}\label{sability_W} Let $(M,g)$ be a compact Riemannian manifold with dimension $n>4$. If the Riemannian universal cover of $(M,g)$ is a rank one symmetric space then there exists an $\epsilon>0$ such that
$$H_w(h,h)\geq \epsilon\|h\|^2  \ \ \forall h\in \dg^{-1}(0)\cap tr^{-1}(0)$$ 
where $H_w$ denotes the Hessian of $\mathcal{W}_{\frac{n}{2}}.$
\end{thm}
Hence Rank one symmetric spcaes are strict local minima for $\mathcal{W}_{\frac{n}{2}.}$ Rigidity and stability of $\mathcal{W}_2$ in dimension $4$ have been studied in \cite{GV} and \cite{KO}.
\section{Notations and Definitions}
Let $\langle,\rangle$,  $\langle,\rangle_{L^2}$ denote the point-wise inner-product and the $L^2$-norm on tensor fields induced by $g$. $|.|$ and $\|.\|$ denotes the corresponding norms. Let $\{e_i\}$ be an orthonormal frame. $\Rt$ is a symmetric 2-tensor defined by
$$\Rt(x,y)=\sum_{i,j,k} R(x,e_i,e_j,e_k)R(y,e_i,e_j,e_k).$$
$\Rc$ is a symmetric operator on $\mathcal{S}^2M$ defined by
$$\Rc(h)(x,y):=\sum R(e_i,x,e_j,y)h(e_i,e_j).$$
Let $R_1,R_2\in S^2(\Lambda^2M)$. Then $R_1\circ R_2\in S^2(\Lambda^2M)$ is defined by
$$R_1\circ R_2(x,y,x,y)=\sum_{i,j} R_1(x,y,e_i,e_j)R_2(x,y,e_i,e_j)$$
The inner product on $S^2(\Lambda^2M)$ is defined by
$$\langle R_1,R_2\rangle= \sum R_1(e_i,e_j,e_k,e_l)R_2(e_i,e_j,e_k,e_l).$$
For any $h_1,h_2\in \mathcal{S}^2M$ define $h_1\wedge h_2 \in S^2(\Lambda^2M)$ by
\Bea h_1\wedge h_2(x,y,z,w)&=&\frac{1}{2}[h_1(x,z)h_2(y,w)+h_1(y,w)h_2(x,z)\\
&&-h_1(x,w)h_2(y,z)-h_1(y,z)h_2(x,w)].
\Eea
Let $D$ denote the Riemannian connection of $(M,g)$. Define the divergence operator $\dg$ on $\mathcal{S}^2M$ by
$$\dg(h)(x)=-\sum_i D_{e_i}h(e_i,x) \ {\rm for} \ h\in \mathcal{S}^2M.$$
If $\dg^*$ denotes its formal adjoint then  
$$\dg^*\dg(h)(x,y)=-\frac{1}{2}\{D^2_{x,e_i}h(e_i,y)+D^2_{y,e_i}h(e_i,x)\}$$
$d^D$ is defined by
$$d^Dh(x,y,z)=D_xh(y,z)-D_yh(x,z)$$
Let $\delta^D$ be its formal adjoint. We have the following formula from \cite{BE}.
\be \delta^Dd^Dh=2D^*Dh-2\delta_g^*\delta_g h+r\circ h+h\circ r-2\Rc(h)\ee
Let $h$ be a {\it TT} tensor. Define,
\Bea &&\la h=D^*Dh\\
&&\la_Lh=\la h+r\circ h+h\circ r-2\Rc(h)\\
&&L_h(x,y,z,w)=D^2_{y,z}h(x,w)+D^2_{x,w}h(y,z)-D^2_{x,z}h(y,w)-D^2_{y,w}h(x,z)
\Eea
Let $g_t$ be a one-parameter family of metrics with $g_0=g$ and $\frac{\partial}{\partial t}g_{t|t=0}=h$. The Riemannian connection and curvatures evolve along $g_t$ at $t=0$ as the following \cite{BA}.
\bea &&C_h(x,y,z)=\frac{\partial}{\partial t}g(D_t(x,y),z)_{t=0}=\frac{1}{2}\{D_xh(y,z)+D_yh(x,z)-D_zh(x,y)\}\\
&& R'_g(h)=\frac{1}{2}[L_h+R\circ h\wedge g]\\
&&r'_g(h)=\frac{1}{2}[D^*Dh+r\circ h+h\circ r-2\Rc(h)-2\dg^*\dg h-Ddtrh]
\eea
\section{Second variations of $\mathcal{R}_2$ at Einstein metrics}
Let $g$ be an Einstein critical metric of $\mathcal{R}_2$. In this section we compute $H_2$ restricted to  transverse traceless variations at an Einstein critical metric of $\mathcal{R}_2$. 
\begin{thm}\label{HessR} Let $(M,g)$ be an Einstein critical metric of $\mathcal{R}_2$ with unit volume and $\lambda$ be its Einstein constant. If $h\in \dg^{-1}(0)\cap tr^{-1}(0) $ then 
\bea H_2(h,h)&=& 2\langle \la h-2\Rc(h), \la h+\lambda h-\Rc(h)\rangle_{L^2}+\frac{2}{n}|R|^2\|h\|^2 \\
&&\nonumber+\langle R\circ h\wedge h-L_h\circ h\wedge g,R\rangle_{L^2}
\eea
\end{thm}
\begin{proof} Since $g$ is a critical metric of $\mathcal{R}_2$
$$H_2(h,h)= \langle (\nabla\mathcal{R}_2)'_g(h),h\rangle_{L^2}.$$
We use the formula for $\nabla\mathcal{R}_2$ in (1.2) to compute $H_2.$ We first prove that
\bea && \langle(\delta^Dd^D r)'_g(h),h\rangle_{L^2}=\langle \la h-2\Rc(h),\la h+\lambda h-\Rc(h)\rangle_{L^2}\\
     && \langle (\Rt)'(h)_g,h\rangle_{L^2}= \frac{1}{2}\langle L_h\circ h\wedge g- R\circ h\wedge h,R\rangle_{L^2}
     \eea
Let $g(t)$ be a one parameter family of metrics for $t\in(-\epsilon,\epsilon)$ with $g(0)=g$. Consider an orthonormal frame with respect to $g$ and fix it.\\
{\it Proof of (3.1):} Since $g$ is Einstein $(\delta^D)'_g(d^Dr)=0$. Hence,
\bea\langle (\delta^Dd^Dr)_g'(h),h\rangle_{L^2} = \langle (d^D)'_g(h)r,d^Dh\rangle_{L^2}+\langle r_g'(h),\delta^Dd^Dh\rangle_{L^2}
\eea
Let $h_1$ be a fixed symmetric $2$-tensor. 
$$d^Dh_1(x,y,z)=x.h_1(y,z)-y.h_1(x,z)-h_1([x,y],z)+h_1(x,D_yz)-h_1(y,D_xz).$$
Therefore,
\Bea (d^D)_g'(h)r(x,y,z)&=& r(D'_g(h)(y,z),x)-r(y,D'_g(h)(x,z))\\
&=& -\lambda (C_h(y,z,x)-C_h(x,z,y))\\
&=&-\lambda (D_xh(y,z)-D_yh(x,z))\\
&=& -\lambda d^Dh(x,y,z)
\Eea
Hence from (3.4) we have, 
\Bea\langle (\delta^Dd^Dr)_g'(h),h\rangle_{L^2} &=& -\lambda\langle d^Dh ,d^Dh\rangle_{L^2}+\frac{1}{2}\langle \la_L,\delta^Dd^Dh\rangle_{L^2}\\
&=&\langle \la h-2\Rc(h),\la h+\lambda h-\Rc(h)\rangle_{L^2}
\Eea
Hence (3.2) follows.\\
{\it Proof of (3.3):}
\Bea \Rt_{pq}=g^{i_1i_2}g^{j_1j_2}g^{k_1k_2}R_{pi_1j_1k_1}R_n{qi_2j_2k_2}
\Eea
Differentiating each terms and using
$$(g^{ij})'=-g^{im}h_{mn}g^{nj}$$ we have,
\Bea (\Rt_g)'(h)_{pq}&=&-\sum_{m,n,i,j}h_{mn}\left(R_{pmij}R_{qnij}+R_{pimj}R_{qinj}+R_{pijm}R_{qijn}\right)\\
&&+\sum_{i,j,k}[(R'_g.h)_{pijk}R_{qijk}+R_{pijk}(R'_g.h)_{qijk}]\\
&=&-\sum_{m,n,i,j}h_{mn}[R_{pmij}R_{qnij}+2 R_{pimj}R_{qinj}]\\
&&+\sum_{i,j,k}[R'_{pijk}R_{qijk}+R_{pijk}R'_{qijk}]
\Eea
\Bea \sum(R'_{pijk}R_{qijk}+R_{pijk}R'_{qijk})h_{pq}&=&2\sum h_{pq}R'_{pijk}R_{qijk}\\
&=& \sum \{L_{pijk}h_{pq}R_{qijk}+2 R_{pimj}R_{qinj}h_{mn}\}
\Eea
Therefore,
\Bea \langle (\Rt)'(h),h\rangle&=& \sum (\Rt_g)'(h)_{pq}h_{pq}\\
&=&\sum \{L_{pijk}R_{qijk}h_{pq}-h_{mn}h_{pq}R_{pmij}R_{qnij}\}
\Eea
By a simple computation we have,
\Bea &(i)&\langle R\circ h\wedge h,R\rangle=2\sum h_{pq}h_{mn}R_{pmij}R_{qnij}\\
&(ii)& \sum (L\circ h\wedge g)_{qijk}=2\sum h_{pq}L_{pijk}
\Eea
Hence,
\Bea \langle (\Rt)'(h),h\rangle_{L^2}
=\frac{1}{2}\langle L\circ h\wedge g,R\rangle_{L^2}-\frac{1}{2}\langle R\circ h\wedge h,R\rangle_{L^2}
\Eea
Proof of Theorem \ref{HessR} follows from (3.2) and (3.3) using (1.2).
\end{proof}
\noindent{\it Remark 1 :} Let $(M,g)$ be a compact flat manifold and $h$ be a TT-tensor then Theorem \ref{HessR} implies that
$$H_2(h,h)=2\|\la h\|^2.$$ 
Therefore, $H_2(h,h)=0$ if and only if $h$ is parallel. The stability condition for $H_2$ is satisfied on TT-tensors orthogonal to the space of parallel tensors.
\\   
{\it Remark 2 :} If $h$ is an infinitesimal Einstein variation of $g$ then $\la h-2\Rc(h)=0.$ Hence the first term of (3.1) vanishes. Since $L_h$ contains various components of $D^2h$, the sign of $H_2(h,h)$ is not clear from the above theorem. On the other hand if $\la h+ \lambda h-\Rc(h)=0$ then $\delta^D d^Dh=0$ i.e. $h$ is a Codazzi tensor. On compact manifolds Codazzi tensors are parallel.
\begin{cor} If $(M,g)$ is an Einstein critical metric of $\mathcal{R}_2$ and $h$ is a parallel TT-tensor then
$$H_2(h,h)=\frac{4}{n}\|R\|^2\|h\|^2$$ 
\end{cor}
\begin{proof} Let $h$ is a parallel TT-tensor. 
$$ D^2_{x,y}h(z,w)-D^2_{y,x}h(z,w)=h(R(x,y,z),w)+h(R(x,y,w),z)$$
Putting $y=z=e_i$ and taking sum over $i$ we have,
\be \Rc(h)=\lambda h \ee
Therefore the first term of the expression of $H_2$ is zero. 
Define, $$h(R)(x,y,z,w)=h(R(x,y,z),w)$$
\Bea 4\langle D^2h,h(R)\rangle&=&4\sum D^2h_{ijkl}h(R)_{ijkl}\\
&=& 2\sum (D^2h_{ijkl}-D^2h_{jikl})h(R)_{ijkl}\\
&=&2 \sum\{h(R)_{ijkl}+h(R)_{ijlk}\}h(R)_{ijkl}\\
&=&2\sum \{ h_{ml}R_{ijkm}h_{nl}R_{ijkn}+h_{mk}R_{ijlm}h_{nl}R_{ijln}\}\\
&=& 2\sum \{\check{R}_{mn}h_{ml}h_{nl}-R_{mnij}R_{lkij}h_{ml}h_{nk}\}\\
&=&\frac{2|R|^2}{n}|h|^2-\langle R\circ h\wedge h,R\rangle
\Eea
If $Dh=0$ then 
$$ \langle R\circ h\wedge h,R\rangle=\frac{2|R|^2}{n}|h|^2 $$
$L_h$ also vanishes in this case by definition. Hence the proof follows.
\end{proof}
\noindent {\it Remark 3 :} If $(M,g)$ is not a flat manifold then $H_2$ restricted to the space of parallel TT-tensors satisfies the stability condition. There are many example of Einstein critical metrics of $\mathcal{R}_2$ which admit parallel TT-tensors. For example let $(N,g_1)$ be an Einstein critical metric of $\mathcal{R}_2$. Then the product metric $(N\times N, g_1\times g_2)$ with $g_2=g_1$ is also a critical metric of $\mathcal{R}_2$ and $g_1-g_2$ is a parallel TT-tensor on $N\times N$.
 
\section{Stability of $\mathcal{R}_2$ at Rank $1$ symmetric spaces}
Let $(M,g)$ be a rank $1$ symmetric space which is not a space form. Then each point has a neighbourhood $U$ such that every tangent space admits complex structures $\{J_\alpha, \alpha=1,2,..,\tau\}$ with the following properties,
\Bea 
&& J_\alpha J_\beta=-J_\beta J_\alpha \ {\rm for} \ \alpha\neq \beta\\
&&J_\alpha J_\beta(x)\in{\rm span}\big( J_0(x),J_1(x),...,J_\tau(x)\big)
\Eea
where $J_0$ is the identity. $\tau=1,3,7$ for $\mathbb{C}P^m$, $\mathbb{H}P^m$ ,$\mathbb{O}P^2$ respectively. There is an orthonormal frame of the form $\{e_{\alpha j}\}$ in $U$ where $e_{\alpha j}=J_{\alpha}e_j$. The components of $R$ with respect to this orthonormal basis $R$ are the following
\Bea &&R(e_{\alpha i},e_{\beta j},e_{\gamma k})=0 \ {\rm if} \  k\neq i\neq j\\
&& R(e_{\gamma i},e_{\alpha i},e_{\beta i})=0  \ {\rm if} \ \alpha\neq\beta\neq\gamma \\
&& R(e_{\alpha i},e_{\beta j},e_{\alpha i},e_{\beta j})=c \ {\rm when} \ \{\alpha,i\}\neq \{\beta,j\}\\
&&R(e_{\alpha i},e_{\beta i},e_{\alpha j},e_{\beta j})=2c \ {\rm if} \ i\neq j \ {\rm and} \ \alpha\neq\beta\\
&&R(e_{\alpha i},e_{\beta i},e_{\alpha i},e_{\beta i})=4c \ {\rm for} \ \alpha\neq\beta\\
&&R(e_{\alpha i},e_{\alpha j},e_{\beta i},e_{\beta j})=c \ {\rm for} \ i\neq j
\Eea
where $c$ is a non-zero constant. Since $J_{\alpha}$ acts isometrically for any $x,y,z,w$
$$R(x,y,J_\alpha z,J_{\alpha}w)=R(x,y,z,w)$$ 
For any $T\in S^2(\Lambda^2 M)$ define,
$$\omega_{\gamma}(x,y)=g(x,J_{\gamma}y)$$
and
$$\Lambda=\sum_{\gamma\neq 0}\omega_{\gamma}\otimes \omega_{\gamma}$$
$\omega_\gamma$ may not be defined globally but $\Lambda$ is globally defined and $D\Lambda=0.$
\begin{lem} Let $T$ be an algebraic curvature tensor. Then
$$\langle R,T\rangle=c\langle T,g\wedge g\rangle+3c\langle T, \Lambda\rangle$$
\end{lem} 
\begin{proof} 
\Bea \langle R,T\rangle &=& \sum_{i,j,k,l}R(e_{\delta i},e_{\delta j},J_{\sigma}e_{\delta k},J_{\sigma}e_{\delta l})T(e_{\delta i},e_{\delta j},J_{\sigma}e_{\delta k},J_{\sigma}e_{\delta l})
\\
&&+\sum_{\delta\neq \gamma}R(e_{\delta i},e_{\gamma j},J_{\sigma}e_{\delta k},J_{\sigma}e_{\gamma l})T(e_{\delta i},e_{\gamma j},J_{\sigma}e_{\delta k},J_{\sigma}e_{\gamma l})\\
&=& A_1^\sigma+A_2^\sigma
\Eea
\Bea
 A_2^0&=& 2\sum R(e_{\delta i},e_{\delta j},e_{\delta i},e_{\delta j}) T(e_{\delta i},e_{\delta j},e_{\delta i},e_{\delta j})\\
 &=& 2c\sum T(e_{\delta i},e_{\delta j},e_{\delta i},e_{\delta j})\\
 \\ 
A_1^0&=& \sum_{\delta\neq \gamma}R(e_{\delta i},e_{\gamma j},e_{\delta k},e_{\gamma l})T(e_{\delta i},e_{\gamma j},e_{\delta k},e_{\gamma l})\\
&=& 2c\sum_{\gamma\neq \delta,i\neq j}\{T(e_{\delta i},e_{\gamma i},e_{\delta j},e_{\gamma j})+T(e_{\delta i},e_{\gamma j},e_{\delta i},e_{\gamma j})\}\\
&&+4c \sum_{\delta \neq\gamma}T(e_{\delta i},e_{\gamma i},e_{\delta i},e_{\gamma i})\\
&=&2c\sum_{\gamma\neq \delta}\{T(e_{\delta i},e_{\gamma i},e_{\delta j},e_{\gamma j})+T(e_{\delta i},e_{\gamma j},e_{\delta i},e_{\gamma j})\}
\Eea
Therefore,
\Bea \langle R,T\rangle &=&2c\sum_{\gamma\neq \delta}\{T(e_{\delta i},e_{\gamma i},J_{\sigma}e_{\delta j},J_{\sigma}e_{\gamma j})+T(e_{\delta i},e_{\gamma j},J_{\sigma}e_{\delta i},J_{\sigma}e_{\gamma j})\}\\
&&+2c\sum T(e_{\delta i},e_{\delta j},J_{\sigma}e_{\delta i},J_{\sigma}e_{\delta j})\\
&=&c\langle T,g\wedge g\rangle+2c\sum_{\sigma\neq 0}T(e_{\delta i},e_{\gamma j},J_{\sigma}e_{\delta i},J_{\sigma}e_{\gamma j})\\
&&+2c\sum T(e_{\delta i},e_{\gamma i},J_{\sigma}e_{\delta j},J_{\sigma}e_{\gamma j})\\
&=&c\langle T,g\wedge g\rangle+3c\sum T(e_{\delta i},e_{\gamma i},J_{\sigma}e_{\delta j},J_{\sigma}e_{\gamma j})\\
&=& c\langle T,g\wedge g\rangle+3c\langle T,\Lambda\rangle
\Eea
Hence the proof follows.
\end{proof}
\begin{lem} If $(M,g)$ is a rank one symmetric space then
$$\langle L\circ h\wedge g,\Lambda\rangle_{L^2}= 0$$
\end{lem}
\begin{proof}
\Bea \langle L\circ h\wedge g,\Lambda\rangle &=& 2\sum L_{qijk}h_{pq}\Lambda_{pijk}\\
&=& 2\sum [D^2_{e_q,e_j}h(e_i,e_k)+D^2_{e_i,e_k}h(e_q,e_j)-D^2_{e_q,e_k}h(e_i,e_j)\\
&&-D^2_{e_i,e_j}h(e_q,e_k)]h(e_p,e_q)\Lambda(e_p,e_i,e_j,e_k)
\Eea
Since $\Lambda$ is parallel,
$$\sum_{i,k} D^2_{e_q,e_j}h(e_i,e_k)\Lambda(e_p,e_i,e_j,e_k)=D^2_{e_q,e_j}\Lc(h)(e_p,e_j)$$
and
$$\sum_j\{ D^2_{e_q,e_j}\Lc(h)(e_p,e_j)+D^2_{e_p,e_j}\Lc(h)(e_q,e_j)\}=\dg^*\dg \Lc(h)(e_p,e_q)$$
Since $\dg h=0$,
$$\langle \dg^*\dg \Lc(h),h\rangle_{L^2}=0$$
Therefore,
\Bea \langle L\circ h\wedge g,\Lambda\rangle_{L^2} &=& 2\langle D^2h,\Lambda_h\rangle_{L^2}\\
&=& 2\langle Dh,D^*\Lambda_h\rangle_{L^2}
\Eea
Define $\dg^{\gamma}h(x)=-\sum D_{e_i}h(J_{\gamma}e_i,x).$
\Bea \langle Dh, D^*\Lambda_h\rangle
&=& 2\sum Dh(e_p,e_q,e_k)\Lambda(e_p,e_q,e_i,e_j)Dh(e_i,e_j,e_k)\\
&=& 2c\sum D_{e_{\alpha i}}h(J_\gamma e_{\alpha i},e_{\delta k})D_{e_{\beta j}}h(J_{\gamma}e_{\beta j},e_{\delta k})\\
&=& 2c\sum\dg^{\gamma} h(e_{\delta k})^2\\
&=&2c|\dg^{\gamma} h|^2\\
&=& 0\Eea
Therefore,
$$\langle L\circ h\wedge g,\Lambda\rangle_{L^2}=0$$
\end{proof}
Define, $\tilde{h}(x,y)=\sum_{\alpha} h(J_{\alpha}x,J_{\alpha}y)$ for $h\in \mathcal{S}^2M.$ 
 $\tilde{h}$ is $J_{\alpha}$ invariant for all $\alpha$ and globally defined.
\begin{thm}\label{HessR_rank1} Let $(M,g)$ be a rank $1$ symmetric space and $h\in \dg^{-1}(0)\cap tr^{-1}(0).$ Then
\Bea H_2(h,h)&=& 2\|\la h-\frac{3}{2}\Rc(h)\|^2+2c(n+3\tau-3)\|Dh\|^2\\
&&+4c^2(3n\tau+3\tau^2+3\tau-4)\|h\|^2+\frac{c^2}{\tau+1}(36n+44-\frac{9}{2}(\tau+1))\|\tilde{h}\|^2
\Eea
\end{thm}
\begin{proof} We compute terms appearing in Theorem \ref{HessR}. 
\Bea
&& (i) \ \Rc(h)=3c\tilde{h}-4ch+c tr(h)g\\
&& (ii) \ \lambda=(n-1+3\tau)c\\
&& (iii) \ |R|^2=2nc^2\{3n\tau+n+3\tau^2+6\tau-1 \}\\
&& (iv) \ \langle (\Rt)'(h)_g,h\rangle_{L^2} = 2c\|Dh\|^2+2c^2(n+5+3\tau)\|h\|^2-\frac{2c^2}{\tau+1}(9n+8)\|\tilde{h}\|^2
\Eea
%------{\it Proof of (i):} -----------------
{\it Proof of (i):} Let $e_1$ be a unit vector. Extend $e_1$ to a basis $\{e_{\alpha i} \}$.
\Bea \Rc(h)(e_1,e_1)
&=&\sum_i R(e_1,e_i,e_1 ,e_i)h(e_i,e_i )+ \sum_{\alpha,i} R(e_1,e_{\alpha i},e_1 ,e_{\alpha i})h(e_{\alpha i},e_{\alpha i} )\\
&=&c\sum_{i\neq 1}[h(e_i,e_i)+h(e_{\alpha i},e_{\alpha i})]
+4c\sum_{\alpha>0} h(e_{\alpha 1},e_{\alpha 1})\\
&=& 3c\tilde{h}(e_1,e_1)+ctr(h)-4ch(e_1,e_1)
\Eea
Consequently we have (ii) as 
$$\Rc(g)=r=\lambda g.$$ 
%---------{\it Proof of (ii) :}--------------------------
{\it Proof of (iii):} 
\Bea \langle R, g \wedge g \rangle 
= 2\langle r, g\rangle
= 2nc(3\tau +n-1)\Eea
\Bea
\langle R, \Lambda \rangle &=&\sum R(e_{\delta i},e_{\gamma i},J_{\sigma}e_{\delta j},J_{\sigma}e_{\gamma j})\\
&=& \sum R(e_{\delta i},e_{\gamma i},J_{\sigma}e_{\delta i},J_{\sigma}e_{\gamma i})+\sum_{i\neq j}R(e_{\delta i},e_{\gamma i},J_{\sigma}e_{\delta j},J_{\sigma}e_{\gamma j})\\
&=& 2cn\tau(n+\tau+1)
\Eea
Now (iii) follows from Lemma 1.\\
%-------------{\it Proof of (iii):} --------------
{\it Proof of (iv):} Since $R\circ h\wedge h-L_h\circ h\wedge g$ is an algebraic curvature tensor from Lemma 1 and 2 we have,
\Bea 2\langle (\Rt)'(h)_g,h\rangle &=& \langle L_h\circ h\wedge g-R\circ h\wedge h,R\rangle\\
&=& c\langle L_h\circ h\wedge g, g\wedge g\rangle-c\langle R\circ h\wedge h, g\wedge g\rangle-3c\langle R\circ h\wedge h, \Lambda\rangle
\Eea
 Let $\{v_i\}$ be an orthonormal basis. 
\Bea
\sum_i L_h(x,v_i,y,v_i)&=& \sum_i [D^2_{v_i,y}h(x,v_i)+D^2_{x,v_i}h(y,v_i)\\
&&-D^2_{x,y}h(v_i,v_i)-D^2_{v_i,v_i}h(x,y)]\\
&=&\sum_i [D^2_{v_i,y}h(x,v_i)+D^2_{x,v_i}h(y,v_i)]\\
&&-Ddtrh(x,y)+D^*Dh(x,y)\\
&=& \la h(x,y)+2\dg^*\dg h(x,y)\\
&&+\sum_i \{D^2_{v_i,y}h(x,v_i)-D^2_{y,v_i}h(x,v_i)\}
\Eea
Using Ricci identity we have,
$$\sum_i \{D^2_{v_i,y}h(x,v_i)-D^2_{y,v_i}h(x,v_i)\}=\lambda h(x,y)-\Rc(h)(x,y).$$
Therefore,
\Bea
\langle L_h\circ h\wedge g, g\wedge g\rangle &=& 2\sum L(v_i,v_j,v_k,v_l)h\wedge g(v_i,v_j,v_k,v_l)\\
&=&4\sum L(v_i,v_j,v_k,v_j)h(e_i,e_k)\\
&=&4\{\langle \la h,h\rangle+\lambda|h|^2-\langle \Rc(h),h\rangle\}
\Eea 
\Bea \langle R\circ h\wedge h, g\wedge g\rangle&=&2\sum R(v_i,v_j,v_k,v_l)h\wedge h(v_i,v_j,v_k,v_l)\\
&=& 2\sum R(v_i,v_j,v_k,v_l)[h(v_i,v_k)h(v_j,v_l)-h(v_i,v_l)h(v_j,v_k)]\\
&=&4\langle\Rc(h),h\rangle
\Eea
Now we consider a basis of the form $\{e_{\alpha i}\}.$
\Bea \langle R\circ h\wedge h, \Lambda\rangle 
&=&\sum R\circ h\wedge h( J_{\sigma}e_{\alpha j},J_{\sigma}e_{\beta j},e_{\alpha i},e_{\beta i})\\
&=&\sum R(J_{\sigma}e_{\alpha j},J_{\sigma}e_{\beta j},J_{\gamma}e_{\alpha k},J_{\gamma}e_{\beta k})h\wedge h(e_{\alpha i},e_{\beta i},J_{\gamma}e_{\alpha k},J_{\gamma}e_{\beta k})\\
&=&6cm(\tau+1)\sum h\wedge h(e_{\alpha i},e_{\beta i},J_{\gamma}e_{\alpha k},J_{\gamma}e_{\beta k})\\
&=& 12mc(\tau+1)\sum h(e_{\alpha i}, J_{\gamma}e_{\alpha k})h(e_{\beta i}, J_{\gamma}e_{\beta k})\\
&=&12m(\tau +1)c\sum \tilde{h}(e_i,e_{\gamma j})\tilde{h}(e_i,e_{\gamma j})\\
&=& 12mc|\tilde{h}|^2
\Eea
Therefore,
\Bea \langle L_h\circ h\wedge g- R\circ h\wedge h,R\rangle_{L^2}
&=&4c\|Dh\|^2+4c\lambda\|h\|^2-8c\langle \Rc(h),h\rangle\\
&&-36mc^2\|\tilde{h}\|^2
\Eea
\Bea \langle \Rc(h),h\rangle &=& 4c\langle \tilde{h},h\rangle-3c|h|^2\\
&=& \frac{4c}{\tau+1}|\tilde{h}|^2-3c|h|^2
\Eea
Therefore,
\Bea \langle L_h\circ h\wedge g-R\circ h\wedge h,R\rangle_{L^2}
&=& 4c\|Dh\|^2+4c^2(n+5+3\tau)\|h\|^2\\
&&-4c^2(9m+\frac{8}{\tau+1})\|\tilde{h}\|^2
\Eea
Hence (iv) follows. Next combining (i),(ii),(iii), (iv) and using Theorem \ref{HessR} we have,
\Bea H_2(h,h)&=& 2\|\la h-\frac{3}{2}\Rc(h)\|^2+2c(n+3\tau-3)\|Dh\|^2\\
&&+4c^2(3n\tau+3\tau^2+3\tau-4)\|h\|^2+\frac{c^2}{\tau+1}(36n+44-\frac{9}{2}(\tau+1))\|\tilde{h}\|^2
\Eea
\end{proof}
\begin{thm}\label{stabilityRTT} Let $(M^n,g)$ be a closed Riemannian manifold. If the Riemannian universal cover of $(M,g)$ is a rank one symmetric space and $p\geq 2$ then there exists an $\epsilon(n,p,R_g)>0$ such that 
$$H_p(h,h)\geq \epsilon \|h\|^2 \ \ \forall h\in \dg^{-1}(0)\cap tr^{-1}(0).$$
\end{thm}
\begin{proof} If $(M,g)$ is a rank one symmetric space of compact type then the theorem follows immediately from Theorem \ref{HessR_rank1} as $c>0$ for them.

If $(M,g)$ is a compact quotient of a rank one symmetric space of non-compact type then we can rewrite the formula for $H_2$ as follows.
\bea \nonumber H_2(h,h)&=& 2\|\la h-\frac{3}{2}\Rc(h)+\lambda h\|^2-2\lambda\langle \la h+\lambda h-\Rc(h), h\rangle_{L^2}\\
&&\nonumber+4c^2(3n\tau+3\tau^2+3\tau-4)\|h\|^2
+\frac{c^2}{\tau+1}(36n+44-\frac{9}{2}(\tau+1))\|\tilde{h}\|^2\\
&=& 2\|\la h-\frac{3}{2}\Rc(h)+\lambda h\|^2-2\lambda\langle \delta^D d^Dh, h\rangle_{L^2}
\\
&&\nonumber+4c^2(3n\tau+3\tau^2+3\tau-4)\|h\|^2+\frac{c^2}{\tau+1}(36n+44-\frac{9}{2}(\tau+1))\|\tilde{h}\|^2
\eea
Since the Einstein constant is negative and $\delta^Dd^D$ is a non-negative operator  
$$H_2(h,h)\geq 4c^2(3n\tau+3\tau^2+3\tau-4)\|h\|^2$$
From \cite{SM1} we observe that if $(M,g)$ is a symmetric space and $h\in \delta_g^{-1}(0)\cap tr^{-1}(0)$ then
$$H_p(h,h)=p|R|^{p-2}H_2(h,h) $$
Hence the proof follows.
\end{proof}
To study $H_p$ on the space of conformal variations of $g$ we recall the following result from \cite{SM2}.
\begin{prop}\cite{SM2} Let $(M,g)$ be a compact irreducible symmetric space and $f\in C^{\infty}(M)$. Then
\Bea H_p(fg,fg)=p|R|^{p-2}[a\|\la f\|^2-b\|df\|^2+c\|f\|^2]
\Eea
where
\Bea &&a= n-1+(p-2)\frac{4s^2}{n^2|R|^2}\\
     &&b=4(p-1)\frac{s}{n}\\
     && c=(p-\frac{n}{2})|R|^2.
\Eea
\end{prop}
\begin{thm}\label{stabilityRconf} Let $(M^n,g)$ be compact rank $1$ symmetric space with dimension $n$ and $p\leq n$. $(M^n,g)$ is stable for $\mathcal{R}_p$ restricted to the conformal variations of $g$ in the following cases.

(i) $\mathbb{C}P^m$ and $p\geq 2.$

(ii) $\mathbb{H}P^m$ and $p\geq 2m.$

(iii) $\mathbb{O}P^2$ and $p\geq 9.$
\end{thm}
\begin{proof}
Let $\mu$ denote the first positive eigenvalue of $\la.$ Since $\frac{\mu}{\lambda}$ and $\frac{|R|}{\lambda}$ are scale-invariant we consider the polynomial 
$$P(x)=ax^2-\frac{b}{\lambda}x+\frac{c}{\lambda^2}$$
If $f$ is an eingenfunction corresponding to the eigenvalue $\tilde{\mu}$ then
$$H_p(fg,fg)= P(\tilde{\mu})\|f\|^2$$
If $p\leq n$, $P(x)$ is increasing for $x\geq \frac{\mu}{\lambda}.$ Therefore,  it is sufficient to prove $P(\frac{\mu}{\lambda})>0$. Using values of $\mu$ listed in \cite{UH} we obtain the following values of $p$ for which $P(\frac{\mu}{\lambda})$ is positive.
\begin{center}
    \begin{tabular}{| l | l | l | l | l | l |}
    \hline
    $(M,g)$ & $\tau$ &  dim & $\frac{|R|^2}{\lambda^2}$ & $\frac{\mu}{\lambda}$ & $p\geq$ \\ 
    \hline
    $\mathbb{C}P^m$ & $1$& $2m$ & $\frac{8m}{m+1}$ & $2$ & $2$ \\ \hline
    $\mathbb{H}P^m$ & $3$& $4m$ & $\frac{2m(10m+11)}{(m+2)^2}$& $\frac{2(m+1)}{m+2}$ & 2m\\ 
    \hline
    $\mathbb{O}P^2$& $7$& $16$ &$\frac{40}{3}$ &$\frac{4}{3}$ & 9\\
    \hline
    \end{tabular}
\end{center}
\end{proof}
Theorem \ref{main} follows from Theorem \ref{stabilityRTT} and Theorem \ref{stabilityRconf}. When $(M,g)$ is a compact quotient of a rank one symmetric space of non-compact type then $s$ is negative. From Proposition 1 it is easy to see that $H_p$ restricted to conformal variations satisfies the stability condition. Therefore proof of Theorem \ref{main2} follows from Theorem \ref{stabilityRconf}.
\section{The Stability of $\mathcal{W}_{\frac{n}{2}}$ at rank $1$ symmetric spaces}
First we compute a formula for the gradient of $\mathcal{W}_p.$ Define
$$\tilde{d}^Dh(x,y,z)=D_yh(x,z)-D_zh(x,y)$$
Let $\tilde{\delta}^D$ be its formal adjoint.
\begin{lem}\label{gradW} $\nabla \mathcal{W}_p=-2\tilde{\delta}^DD^*(|W|^{p-2}W)-|W|^{p-2}[\frac{4}{n-2}\Wc(r)+2\check{W}-\frac{1}{2}|W|^2g]$
\end{lem}
\begin{proof} Following proof of Proposition $4.70$ in \cite{BA} we obtain
$$(\mathcal{W}_p)'_g(h)=p\int_M[2\langle W,W'_g(h)\rangle-4\langle \check{W},h\rangle+\frac{1}{2}trh]|W|^{p-2}dv_g.$$
From (1.3) we have,
$$W'(h)=R'(h)-\frac{2}{n-2}[r'(h)\wedge g+r\wedge h]+\frac{4s}{(n-1)(n-2)}h\wedge g.$$
Putting $R'(g)(x,y,z,w)=D_yC_h(x,z,w)-D_xC_h(y,z,w)+R(x,y,z,h^{\sharp}(w))$ we have,
$$\langle W,W'_g(h)\rangle=-2\langle W,DC_h\rangle+\langle \check{W},h\rangle-\frac{2}{n-2}\langle \Wc(r),h\rangle.$$
Therefore,
 $$(|W|^2)'(g)=-4\langle W,DC_h\rangle-\frac{4}{n-2}\langle \Wc(r),h\rangle-2\langle \check{W},h\rangle.$$
Next the proof follows from the following identity.
$$2\langle D^*W,C_h\rangle=\langle D^*W,\tilde{d}^Dh\rangle.$$
\end{proof}
\begin{lem}\label{HessW} Let $(M,g)$ be a locally symmetric space which is also a critical point of $\mathcal{W}_p.$ If $\tilde{H}_p$ denotes the Hessian of $\mathcal{W}_p$ and $h$ is a TT-tensor then $\tilde{H}_p(h,h)=|W|^{p-2}\tilde{H}_2(h,h)$.
\end{lem}
\begin{proof}
If $(M,g)$ is a locally symmetric space then $\delta^DD^*(|W|^{p-2}W)=0.$ Therefore, if $(M,g)$ is a critical point of $\mathcal{W}_p$ then 
$$\frac{4}{n-2}\Wc(r)+2\check{W}-\frac{1}{2}|W|^2g=0.$$
From the proof of Lemma \ref{gradW} we observe that if $h$ is a TT-tensor then
\Bea (|W|^2)'(h)&=&-4\langle W,DC_h\rangle\\
                &=& 4\delta\dg (\Wc(h))\\
                &=& 0 .
\Eea
Therefore, using Lemma \ref{gradW} we have,
\Bea \tilde{H}_p(h,h)&=& \langle (\nabla \mathcal{W}_p)'(h),h\rangle_{L^2}\\
                     &=& |W|^{p-2}\langle (\nabla \mathcal{W}_2)'(h),h\rangle_{L^2}\\
                     &=& |W|^{p-2}\tilde{H}_2
\Eea
\end{proof}
Since $\mathcal{W}_{\frac{n}{2}}$ is conformally invariant we restrict the Hessian of $\mathcal{W}_{\frac{n}{2}}$ to $\dg^{-1}(0)\cap tr^{-1}(0)$. A critical metric of $\mathcal{W}_{\frac{n}{2}}$ is said to be stable if there exists an $\epsilon>0$ such that 
$$\tilde{H}_{\frac{n}{2}}(h,h)\geq \epsilon \|h\|^2 \ \forall \ h\in \dg^{-1}(0)\cap tr^{-1}(0).$$
\begin{thm} Let $(M^n,g)$ be a closed manifold whose Riemannian universal cover is a rank one symmetric space with dimension $n>4$. Then $(M,g)$ is stable for $\mathcal{W}_{\frac{n}{2}}.$
\end{thm}
\begin{proof} 
Define,
 $$\mathcal{F}(g)=\int_M(|r|^2-\frac{s^2}{2(n-1)})dv_g$$
Using the decomposition in (1.3) we have, 
$$\mathcal{W}_2=\mathcal{R}_2-\frac{4}{n-2}\mathcal{F}.$$ 
Lemma \ref{HessW} implies that it is sufficient to study the Hessian of $\mathcal{W}_2$ denoted by $\tilde{H}_2.$ Let $H_F$ denote the Hessian of $\mathcal{F}.$ Therefore,
$$\tilde{H}_2=H_2-\frac{4}{n-2}H_F$$
From Theorem (3.5) in \cite{GV} we have,
\bea H_F&=& \frac{1}{2}\|\la h-2\Rc(h)\|^2+\frac{n\lambda}{2(n-1)} \|Dh\|^2 -\frac{n\lambda}{n-1}\langle\Rc(h),h\rangle
\eea
{\it Case I :} Let $(M,g)$ be a closed manifold with constant curvature. Then $R=\frac{\lambda}{n-1}g\wedge g$ and $\Rc(h)=-ch$. Theorem \ref{HessR} implies that
\Bea H_2(h,h)&=&2\langle \la h+2ch, \la h+(\lambda+1)c h\rangle+4(n-1)c^2\|h\|^2 \\
&&-4c\{\langle \la h,h\rangle+(\lambda+2c) \|h\|^2\}\\
&=& 2\{\|\la h\|^2+nc|Dh|^2+2(n-2)c^2\|h\|^2\}
\Eea
From (5.1) we have,
\Bea 2H_F(h,h)&=&\|\la h+2ch\|^2+nc\|Dh\|^2 +2nc^2\|h\|^2\\
              &=& \|\la h\|^2+(n+4)c\|Dh\|^2+2(n+2)c^2\|h\|^2
\Eea
Therefore,
\Bea \tilde{H}_2&=& \frac{2}{n-2}[(n-4)\|\la h\|^2+(n^2-2n-8)c\|Dh\|^2+2(n^2-2n-8)c^2\|h\|^2]
\Eea
If $(M,g)$ is a spherical space form with dimension greater than $4$ then $c>0$. Hence it is stable for $\mathcal{W}_{\frac{n}{2}}$. Let $(M,g)$ be a compact hyperbolic manifold and $h$ be a TT-tensor.
 $$\langle \delta^Dd^D h,h\rangle_{L^2}=\|d^Dh\|^2 \geq0$$
Therefore, $ \langle \la h+nch,h\rangle_{L^2} $ is non-negative.
\Bea \tilde{H}_2(h,h)&=&\frac{2}{n-2}\{(n-4)\|\la h\|^2+(n^2-2n-8)c\|Dh\|^22(n^2-2n-8)c^2\|h\|^2\}\\
&=&\frac{2}{n-2}\{ (n-4)\|\la h+nch\|^2-(n^2-6n+8)c \langle \la h+nch,h\rangle_{L^2}\\
&&+4(n-4)c^2\|h\|^2\}\\
&\geq &\frac{8(n-4)c^2}{n-2}\|h\|^2
\Eea
Hence compact hyperbolic manifolds with $n >4$ are stable for $\mathcal{W}_{\frac{n}{2}}.$ \\
\\
{\it Case II : }  Let $(M,g)$ be a rank one symmetric space which is not a space form then 
\Bea\tilde{H}_2&=& \frac{2(n-4)}{n-2}\|\la h-\frac{3n-14}{2(n-4)}\Rc(h)\|^2+\frac{23n-30}{(n-4)(n-2)}\|\Rc(h)\|^2\\
&&+\{2(n-3)+\frac{6(n^2-4n+2)\tau}{(n-1)(n-2)}\}c\|Dh\|^2\\
&&+8c^2(3n\tau+3\tau^2+3\tau-\frac{2(n-4)}{n-2})\|h\|^2\\
&& +\frac{2c^2}{\tau +1}(36n+33.5-4.5\tau-\frac{12}{n-2}-\frac{18n\tau}{(n-2)(n-1)})\|\tilde{h}\|^2\\
&&
\Eea
Let $(M,g)$ be a rank one symmetric space of compact type. Then $c>0$ and the stability of $\mathcal{W}_{\frac{n}{2}}$ follows from the above expression. If $(M,g)$ is a quotient of rank one symmetric space of on-compact type it follows from (4.1) and (5.1).
\end{proof}
As a consequence we have the following theorem using the proof of Proposition (6.4) in \cite{GV}.
\begin{thm}\label{minimaW} Let $(M,g)$ be closed manifold with dimension $n>4$. If the Riemannian universal cover of $(M,g)$ is a rank one symmetric space then $g$ is a strict local minima for $\mathcal{W}_{\frac{n}{2}}.$
\end{thm}
If $(M,g)$ is a spherical space form or a compact hyperbolic manifold then $g$ is a global minima for $\mathcal{W}_{\frac{n}{2}}.$  By Theorem \ref{minimaW} $(M,g)$ there exists a neighbourhood $\mathcal{U}$ of $g$ in $C^{k,\alpha}$-topology such that if $\tilde{g}$ is a conformally flat metric in $\mathcal{U}$ then $\tilde{g}=f\phi^*g$ for some positive smooth function $f$ on $M$ and a diffeomorphism $\phi$ of $M.$ 

\bibliographystyle{amsplain}

\end{document}